\documentclass[12pt]{article}

\usepackage{amsmath, epsfig, cite}
\usepackage{amssymb}
\usepackage{amsfonts}
\usepackage{latexsym}
\usepackage{amsthm}

\newtheorem{thm}{Theorem}[section]

\newtheorem{lem}[thm]{Lemma}

\numberwithin{equation}{section}

\renewcommand{\thefootnote}

\begin{document}

\begin{center}
{\large\bf On two conjectural series involving Riemann zeta function
 \footnote{Corresponding author$^*$. Email addresses: weichuanan78@163.com (C. Wei),
 cexu2020@ahnu.edu.cn (C. Xu).}}
\end{center}

\renewcommand{\thefootnote}{$\dagger$}

\vskip 2mm \centerline{Chuanan Wei$^1$, Ce Xu$^{2*}$}
\begin{center}
{$^1$School of Biomedical Information and Engineering\\
  Hainan Medical University, Haikou 571199, China\\
  $^2$School of Mathematics and Statistics\\
Anhui Normal University, Wuhu 241002, China}
\end{center}


\vskip 0.7cm \noindent{\bf Abstract.} Riemann zeta function is
important in a lot of branches of number theory. With the help of
the operator method and several transformation formulas for
hypergeometric series, we prove four series involving Riemann zeta
function. Two of them are series expansions for $\zeta(7)$ and
$\zeta(3)^2$ recently conjectured by Z.-W. Sun.



\section{Introduction}

For $\ell,n\in \mathbb{Z}^{+}$, define the generalized harmonic
number of $\ell$-order by
\[H_{n}^{(\ell)}(x)=\sum_{k=1}^n\frac{1}{(x+k)^{\ell}}.\]
When $x=0$, it reduce to the
 harmonic number of $\ell$-order
\[H_{n}^{(\ell)}=\sum_{k=1}^n\frac{1}{k^{\ell}}.\]
Taking $\ell=1$ in the last expression, we obtain the classical
harmonic number:
\[
H_{n}=\sum_{k=1}^n\frac{1}{k}.\]

For a nonnegative integer $m$, define the shifted-factorial to be
\begin{align*}
(x)_0=1 \quad\text{and}\quad (x)_m=x(x+1)\cdots(x+m-1) \quad
\text{with} \quad m\in \mathbb{Z}^{+}.
\end{align*}
For two complex sequences $\{a_i\}_{i\geq1}$, $\{b_j\}_{j\geq1}$ and
a complex variable $z$, define the hypergeometric series by

$$
_{r+1}F_{r}\left[\begin{array}{c}
a_1,a_2,\ldots,a_{r+1}\\
b_1,b_2,\ldots,b_{r}
\end{array};\, z
\right] =\sum_{k=0}^{\infty}\frac{(a_1)_k(a_2)_k\cdots(a_{r+1})_k}
{(1)_k(b_1)_k\cdots(b_{r})_k}z^k.
$$
Then a long hypergeometric transformation (cf.\cite[P. 27]{Bailey})
can be stated as
\begin{align}
&{_{9}F_{8}}\left[\begin{array}{cccccccc}
  a,1+\frac{a}{2},b,c,d,e,f,g,-n\\
  \frac{a}{2},1+a-b,1+a-c,1+a-d,1+a-e,1+a-f,1+a-g,1+a+n
\end{array};1\right]
\notag\\[1mm]
&=\frac{(1+a)_n(1+\lambda-e)_n1+\lambda-f)_n1+\lambda-g)_n}{(1+\lambda)_n(1+a-e)_n(1+a-f)_n(1+a-g)_n}
\notag\\[1mm]
&\times{_{9}F_{8}}\left[\begin{array}{cccccccc}
  \lambda,1+\frac{\lambda}{2},\lambda+b-a,\lambda+c-a,\lambda+d-a,e,f,g,-n\\
  \frac{\lambda}{2},1+a-b,1+a-c,1+a-d,1+\lambda-e,1+\lambda-f,1+\lambda-g,1+\lambda+n\end{array};1
\right], \label{9F8}
\end{align}
where $\lambda=1+2a-b-c-d$ and $2+3a=b+c+d+e+f+g-n$.

 For a differentiable function $f(x)$, define the derivative operator
$\mathcal{D}_x$ by
\begin{align*}
\mathcal{D}_xf(x)=\frac{d}{dx}f(x)=\lim_{\bigtriangleup
x\to0}\frac{f(x+\bigtriangleup x)-f(x)}{\bigtriangleup x}.
 \end{align*}
 Sometimes we use the derivative operator $\mathcal{D}_x|_{x=a}$ to denote
\begin{align*}
\mathcal{D}_x|_{x=a}f(x)=\frac{d}{dx}f(x)\bigg|_{x=a}.
 \end{align*}
So it is ordinary to show that
\begin{align*}
&\mathcal{D}_x\:(1+x)_r=(1+x)_rH_r(x),
\\[1mm]
&\:\:\:\mathcal{D}_x|_{x=0}\:(1+x)_r=(1)_rH_r.
 \end{align*}
Some interesting harmonic number identities from differentiation of
the shifted-factorials can be viewed in the papers
\cite{Liu,Paule,Sofo,Wang-Wei}.

For $s\in C$ with $\mathfrak{R}(s)>1$, define the Riemann zeta
function by
$$\zeta(s)=\sum_{k=1}^{\infty}\frac{1}{k^s}.$$
For $\{k_1,\dots,k_d\}\subseteq\mathbb{Z}^{+}$, the Hoffman multiple
$t$-values (cf. \cite{Hoffman2019}) may be defined by
\begin{align*}
t(k_1,\ldots,k_d):=\sum_{\substack{0<n_1<\cdots<n_d\\ n_j:
\text{odd}}}
    \frac{1}{n_1^{k_1}\cdots n_d^{k_d}}=\sum_{\substack{0<n_1<\cdots<n_d}}
    \frac{1}{(2n_1-1)^{k_1}\cdots (2n_d-1)^{k_d}}.
\end{align*}
It is obvious that the Hoffman multiple $t$-values satisfy the
series stuffle relations (cf. \cite{Hoffman2000}). For example,
\begin{align}\label{doubletvalues}
t(k)t(\ell)=t(k,\ell)+t(\ell,k)+t(k+\ell) \quad\text{with}\quad
k,\ell\geq 2.
\end{align}

A surprising series for $\zeta(3)$ owns to Guillera \cite{Guillera}
is
\begin{align}
&\quad\sum_{k=0}^{\infty}\bigg(\frac{-1}{4}\bigg)^k\frac{(1)_k^5}{(\frac{3}{2})_k^5}(10k^2+14k+5)=\frac{7}{2}\zeta(3).
\label{Guillera}
\end{align}
Two nice $\pi$-series, which are respectively due to Zeilberger
\cite{Zeilberger} and Ramanujan \cite{Ramanujan}, can be written as
\begin{align}
&\sum_{k=0}^{\infty}\bigg(\frac{1}{64}\bigg)^k\frac{(1)_k^3}{(\frac{3}{2})_k^3}(21k+13)
=\frac{4\pi^2}{3},
 \label{Zeilberger}\\
&\:\:\sum_{k=0}^{\infty}\bigg(\frac{1}{64}\bigg)^k\frac{(\frac{1}{2})_k^3}{(1)_k^3}(42k+5)
=\frac{16}{\pi}.
 \label{Ramanujan}
\end{align}
 For some
conclusions similar to \eqref{Guillera}-\eqref{Ramanujan}, the
reader is referred to
 the papers \cite{Guo, Guillera-c, Liu-a, Sun-b, Wang, Wang-Sun}.

Recently, Sun \cite[Equations (6.12) and (6.11)]{Sun-d} proposed the
following two conjectures involving $\zeta(7)$ and $\zeta(3)^2$
associated with \eqref{Guillera}.

\begin{thm}\label{thm-a}
\begin{align}
\sum_{k=0}^{\infty}\bigg(\frac{-1}{4}\bigg)^k\frac{(1)_k^5}{(\frac{3}{2})_k^5}(10k^2+14k+5)\Big[16H_{2k+1}^{(4)}+3H_{k}^{(4)}\Big]=\frac{127\zeta(7)}{2}.
 \label{equation-wei-a}
\end{align}
\end{thm}

\begin{thm}\label{thm-b}
\begin{align}
\qquad\sum_{k=0}^{\infty}\bigg(\frac{-1}{4}\bigg)^k\frac{(1)_k^5}{(\frac{3}{2})_k^5}\bigg\{(10k^2+14k+5)\Big[8H_{2k+1}^{(3)}-H_{k}^{(3)}\Big]+\frac{2}{k+1}\bigg\}
=\frac{49\zeta(3)^2}{2}.
 \label{equation-wei-b}
\end{align}
\end{thm}

Sun \cite[Equation (3.5)]{Sun-c} and \cite[Equation (6.2)]{Sun-d}
conjectured the following two series related to \eqref{Zeilberger}:
\begin{align*}
&\sum_{k=0}^{\infty}\bigg(\frac{1}{64}\bigg)^k\frac{(1)_k^3}{(\frac{3}{2})_k^3}(21k+13)\Big\{8H_{2k+1}^{(3)}+43H_{k}^{(3)}\Big\}
=\frac{711}{28}\zeta(3)-\frac{29}{14}\pi^2\zeta(3),\\
&\sum_{k=0}^{\infty}\bigg(\frac{1}{64}\bigg)^k\frac{(1)_k^3}{(\frac{3}{2})_k^3}\bigg\{(21k+13)H_{k}^{(3)}-\frac{1}{(k+1)^2}\bigg\}
=\frac{496}{7}\zeta(3)-\frac{128}{21}\pi^2\zeta(3).
\end{align*}
Inspired by them, we shall prove the following formula.

\begin{thm}\label{thm-c}
\begin{align}
\sum_{k=0}^{\infty}\bigg(\frac{1}{64}\bigg)^k\frac{(1)_k^3}{(\frac{3}{2})_k^3}\bigg\{(21k+13)\Big[H_{2k+1}^{(3)}+2H_{k}^{(3)}\Big]-\frac{27}{8(k+1)^2}\bigg\}
=4\pi^2\zeta(3)-36\zeta(5).
 \label{equation-wei-c}
\end{align}
\end{thm}

Sun \cite[Equation (3.5)]{Sun-c} and \cite[Equation (6.15)]{Sun-d}
conjectured the following two series connected to \eqref{Ramanujan}:
\begin{align*}
&\quad\:\sum_{k=0}^{\infty}\bigg(\frac{1}{64}\bigg)^k\frac{(\frac{1}{2})_k^3}{(1)_k^3}(42k+5)\bigg\{H_{2k}^{(3)}-\frac{43}{352}H_{k}^{(3)}\bigg\}
=\frac{555}{77}\frac{\zeta(3)}{\pi}-\frac{32}{11}G,
\\
&\sum_{k=0}^{\infty}\bigg(\frac{1}{64}\bigg)^k\frac{(\frac{1}{2})_k^3}{(1)_k^3}\bigg\{(42k+5)H_{k}^{(3)}-\frac{352}{(2k+1)^2}\bigg\}
=\frac{10720}{7}\frac{\zeta(3)}{\pi}-\frac{7168}{7}G,
\end{align*}
where $G$ is the Catalan constant
$$G=\sum_{k=1}^{\infty}\frac{(-1)^{k-1}}{(2k-1)^2}.$$
Inspired by them, we shall establish the following theorem.

\begin{thm}\label{thm-d}
\begin{align}
\sum_{k=0}^{\infty}\bigg(\frac{1}{64}\bigg)^k\frac{(\frac{1}{2})_k^3}{(1)_k^3}\bigg\{(42k+5)\Big[17H_{2k}^{(3)}-2H_{k}^{(3)}\Big]-\frac{27}{(2k+1)^2}\bigg\}
=\frac{240\zeta(3)}{\pi}-128G.
 \label{equation-wei-d}
\end{align}
\end{thm}

The rest of the paper is organized as follows. According to the
operator method and two transformation formulas for hypergeometric
series, we shall certify Theorems \ref{thm-a} and \ref{thm-b} in
Sections 2 and 3. Similarly, Theorems \ref{thm-c} and \ref{thm-d}
will be verified in Section 3.

\section{Proof of Theorems \ref{thm-a}}
Now we are ready to prove Theorem \ref{thm-a}.

\begin{proof}[Proof of Theorem \ref{thm-a}]
Recall a  transformation formula for hypergeometric series (cf.
\cite[Theorem 10]{Chu-b}):
\begin{align}
&\sum_{k=0}^{\infty}(-1)^k\frac{(1+a-b-c)_k(1+a-b-d)_{k}(1+a-b-e)_{k}(1+a-c-d)_k}{(1+a-b)_{k}(1+a-c)_{k}(1+a-d)_{k}(1+a-e)_{k}}
\notag\\
&\quad\times\frac{(1+a-c-e)_{k}(1+a-d-e)_{k}}{(1+2a-b-c-d-e)_{2k+2}}A_k(a,b,c,d,e)
\notag\\
&\:=\sum_{k=0}^{\infty}(a+2k)\frac{(b)_k(c)_k(d)_k(e)_k}{(1+a-b)_{k}(1+a-c)_{k}(1+a-d)_{k}(1+a-e)_{k}},
\label{equation-wei-e}
\end{align}
where $\mathfrak{R}(1+2a-b-c-d-e)>0$ and
\begin{align*}
A_k(a,b,c,d,e)&=(1+2a-b-c-d+2k)(2+2a-b-c-d-e+2k)(a-e+k)
\\[1mm]
&\quad+(1+a-b-c+k)(1+a-b-d+k)(1+a-c-d+k).
\end{align*}

The $(a,b,c,d,e)=(1,x,1-x,x,1-x)$ case of \eqref{equation-wei-e}
reads
\begin{align*}
\sum_{k=0}^{\infty}\bigg(\frac{-1}{4}\bigg)^k\frac{(1)_k^3(2x)_k(2-2x)_k}{(\frac{3}{2})_k(1+x)_k^2(2-x)_k^2}
B_k(x)=
2\sum_{k=0}^{\infty}(2k+1)\frac{(x)_k^2(1-x)_k^2}{(2-x)_k^2(1+x)_k^2},
\end{align*}
where
\begin{align*}
B_k(x)=2(x+k)(2-x+2k)+(1+k)(2-2x+k).
\end{align*}
Employ the derivative operator  $\mathcal{D}_{x}$ on both sides of
the last equation to get
\begin{align*}
&\sum_{k=0}^{\infty}\bigg(\frac{-1}{4}\bigg)^k\frac{(1)_k^3(2x)_k(2-2x)_k}{(\frac{3}{2})_k(1+x)_k^2(2-x)_k^2}
\notag\\[1mm]
&\quad\times
\Big\{(1-2x)+B_k(x)\big[H_{k}(2x-1)-H_{k}(1-2x)+H_{k}(1-x)-H_{k}(x)\big]
\Big\}
\notag\\[1mm]
&\:\:=
2\sum_{k=0}^{\infty}(2k+1)\frac{(x)_k^2(1-x)_k^2}{(2-x)_k^2(1+x)_k^2}\Big\{H_{k}(x-1)-H_{k}(-x)+H_{k}(1-x)-H_{k}(x)\Big\}.
\end{align*}
 Dividing both sides by $(1-2x)$, there holds
\begin{align}
&\sum_{k=0}^{\infty}\bigg(\frac{-1}{4}\bigg)^k\frac{(1)_k^3(2x)_k(2-2x)_k}{(\frac{3}{2})_k(1+x)_k^2(2-x)_k^2}
\notag\\[1mm]
&\quad\times\bigg\{1+B_k(x)\bigg[\sum_{i=1}^{k}\frac{2}{(2x-1+i)(1-2x+i)}-\sum_{i=1}^{k}\frac{1}{(x+i)(1-x+i)}\bigg]\bigg\}
\notag\\[1mm]
 &\:\:=
2\sum_{k=0}^{\infty}(2k+1)\frac{(x)_k^2(1-x)_k^2}{(2-x)_k^2(1+x)_k^2}
\notag\\[1mm]
&\quad\times
\bigg\{\sum_{i=1}^{k}\frac{2}{(x-1+i)(-x+i)}-\sum_{i=1}^{k}\frac{1}{(x+i)(1-x+i)}\bigg\}.
  \label{equation-wei-f}
\end{align}
Apply the derivative operator  $\mathcal{D}_{x}|_{x=\frac{1}{2}}$ on
both sides of \eqref{equation-wei-f} to find
\begin{align}
&\sum_{k=0}^{\infty}\bigg(\frac{-1}{4}\bigg)^k\frac{(1)_k^5}{(\frac{3}{2})_k^5}\bigg\{(10k^2+14k+5)
\Big[H_{k}^{(4)}(\tfrac{1}{2})-8H_{k}^{(4)}+2\Big(H_{k}^{(2)}(\tfrac{1}{2})-2H_{k}^{(2)}\Big)^2\Big]
\notag\\[1mm]
&\qquad
 -8\Big(H_{k}^{(2)}(\tfrac{1}{2})-2H_{k}^{(2)}\Big) \bigg\}
 \notag\\[1mm]
&\quad
 =\sum_{k=0}^{\infty}\frac{4}{(2k+1)^3}
 \bigg\{2\Big(H_{k}^{(2)}(-\tfrac{1}{2})-H_{k}^{(2)}(\tfrac{1}{2})\Big)^2-\Big(H_{k}^{(4)}(-\tfrac{1}{2})-H_{k}^{(4)}(\tfrac{1}{2})\Big)
\bigg\}.
 \label{equation-wei-g}
\end{align}

Performing the substitutions $(a,c,d,e)\to(1,c-b,d-c,2-d)$ in
\eqref{equation-wei-e}, we have
\begin{align}
&\sum_{k=0}^{\infty}(-1)^k\frac{(b-c+d)_k(d-b)_{k}(2+b-d)_{k}(2-b+c-d)_k(c)_{k}(2-c)_{k}}{(2-b)_{k}(2+b-c)_{k}(d)_{k}(2+c-d)_{k}(1)_{2k+2}}
C_k(b,c,d)
\notag\\[1mm]
&\:=\sum_{k=0}^{\infty}(2k+1)\frac{(b)_k(c-b)_k(d-c)_k(2-d)_k}{(2-b)_{k}(2+b-c)_{k}(2+c-d)_{k}(d)_{k}},
\label{equation-wei-h}
\end{align}
where
\begin{align*}
C_k(b,c,d)&=2(1+k)(d-1+k)(3-d+2k)
\\[1mm]
&\quad+(2+b-d+k)(2-b+c-d+k)(2-c+k).
\end{align*}
Employ the derivative operator $\mathcal{D}_{b}|_{b=\frac{1}{2}}$ on
both sides of \eqref{equation-wei-h} to deduce
\begin{align*}
&\sum_{k=0}^{\infty}\bigg(\frac{-1}{4}\bigg)^k\frac{(c)_k(2-c)_k(\frac{3}{2}+c-d)_k(\frac{1}{2}-c+d)_k(d-\frac{1}{2})_k(\frac{5}{2}-d)_k}
{(\frac{5}{2}-c)_k(2+c-d)_k(d)_k(2)_k(\frac{3}{2})_k^2}
\notag\\[1mm]
&\quad\times
\Big\{\big[H_{k}(\tfrac{3}{2}-d)-H_{k}(\tfrac{1}{2}+c-d)+H_{k}(d-c-\tfrac{1}{2})-H_{k}(d-\tfrac{3}{2})+H_{k}(\tfrac{1}{2})
\notag\\[1mm]
&\qquad\quad
 -H_{k}(\tfrac{3}{2}-c)\big]C_k(\tfrac{1}{2},c,d)+(c-1)(2-c+k) \Big\}
\notag\\[1mm]
&\:\:=
2\sum_{k=0}^{\infty}\frac{(c-\frac{1}{2})_k(d-c)_k(2-d)_k}{(\frac{5}{2}-c)_k(2+c-d)_k(d)_k}
\Big\{H_{k}(-\tfrac{1}{2})-H_{k}(c-\tfrac{3}{2})+H_{k}(\tfrac{1}{2})-H_{k}(\tfrac{3}{2}-c)\Big\}.
\end{align*}
Dividing both sides by $(c-1)$, there is

\begin{align}
&\sum_{k=0}^{\infty}\bigg(\frac{-1}{4}\bigg)^k\frac{(c)_k(2-c)_k(\frac{3}{2}+c-d)_k(\frac{1}{2}-c+d)_k(d-\frac{1}{2})_k(\frac{5}{2}-d)_k}
{(\frac{5}{2}-c)_k(2+c-d)_k(d)_k(2)_k(\frac{3}{2})_k^2}
\notag\\[1mm]
&\quad\times
\bigg\{\bigg[\sum_{i=1}^{k}\frac{1}{(\frac{3}{2}-d+i)(\frac{1}{2}+c-d+i)}+\sum_{i=1}^{k}\frac{1}{(d-c-\frac{1}{2}+i)(d-\frac{3}{2}+i)}
\notag\\[1mm]
&\qquad\quad
 -\sum_{i=1}^{k}\frac{1}{(\frac{1}{2}+i)(\frac{3}{2}-c+i)}\bigg]C_k(\tfrac{1}{2},c,d)+(2-c+k) \bigg\}
\notag\\[1mm]
&\:\:=
2\sum_{k=0}^{\infty}\frac{(c-\frac{1}{2})_k(d-c)_k(2-d)_k}{(\frac{5}{2}-c)_k(2+c-d)_k(d)_k}
\notag\\[1mm]
&\quad\times
\bigg\{\sum_{i=1}^{k}\frac{1}{(-\frac{1}{2}+i)(c-\frac{3}{2}+i)}-\sum_{i=1}^{k}\frac{1}{(\frac{1}{2}+i)(\frac{3}{2}-c+i)}
 \bigg\}.
 \label{equation-wei-i}
\end{align}
Applying the derivative operator $\mathcal{D}_{c}|_{c=1}$ on both
sides of \eqref{equation-wei-i} and then the derivative operator
$\mathcal{D}_{d}|_{d=\frac{3}{2}}$ on both sides of the resulting
identity, it is not difficult to understand that
\begin{align}
&\sum_{k=0}^{\infty}\bigg(\frac{-1}{4}\bigg)^k\frac{(1)_k^5}{(\frac{3}{2})_k^5}\bigg\{(10k^2+14k+5)
\Big[-6H_{k}^{(4)}+\Big(H_{k}^{(2)}(\tfrac{1}{2})-2H_{k}^{(2)}\Big)^2\Big]
\notag\\[0.5mm]
&\qquad
 -4\Big(H_{k}^{(2)}(\tfrac{1}{2})-2H_{k}^{(2)}\Big) \bigg\}
 \notag\\[0.5mm]
&\quad
 =\sum_{k=0}^{\infty}\frac{4}{(2k+1)^3}
 \Big(H_{k}^{(2)}(-\tfrac{1}{2})-H_{k}^{(2)}(\tfrac{1}{2})\Big)^2.
 \label{equation-wei-j}
\end{align}
Therefore, the linear combination of \eqref{Guillera},
\eqref{equation-wei-g} and \eqref{equation-wei-j} gives
\eqref{equation-wei-a}.
\end{proof}

\section{Proof of Theorem \ref{thm-b}}

In order to prove Theorem \ref{thm-b}, we require the following
lemma.

\begin{lem}\label{lemm-a}
\begin{align*}
\sum_{k=0}^\infty
\frac{8H_{2k}^{(3)}-H_k^{(3)}}{(2k+1)^3}=\frac{49}{16}\zeta(3)^2-\frac{\pi^6}{240}.
\end{align*}
\end{lem}

\begin{proof}
Via the definition of Hoffman double $t$-values, we discover
\begin{align*}
\sum_{k=0}^\infty
\frac{8H_{2k}^{(3)}-H_k^{(3)}}{(2k+1)^3}=8\sum_{k=0}^\infty
\frac{1}{(2k+1)^3}\sum_{j=1}^k \frac1{(2j-1)^3}=8t(3,3).
\end{align*}
The $k=\ell=3$ case of \eqref{doubletvalues} produces
\begin{align*}
t(3,3)=\frac1{2}\big\{t(3)^2-t(6)\big\}=\frac{49}{128}\zeta(3)^2-\frac{63}{128}\zeta(6).
\end{align*}
Hence, we arrive at the desired result.
\end{proof}

Subsequently, we begin to prove Theorem \ref{thm-b}.

\begin{proof}[{\bf{Proof of Theorem \ref{thm-b}}}]
A known transformation formula for hypergeometric series (cf.
\cite[Theorem 9]{Chu-b}) can be expressed as
\begin{align}
&\sum_{k=0}^{\infty}\frac{(c)_k(d)_k(e)_k(1+a-b-c)_k(1+a-b-d)_{k}(1+a-b-e)_{k}}{(1+a-c)_{k}(1+a-d)_{k}(1+a-e)_{k}(1+2a-b-c-d-e)_{k}}
\notag\\[1mm]
&\quad\times\frac{(-1)^k}{(1+a-b)_{2k}}E_k(a,b,c,d,e) \notag
\notag\\[1mm]
&\:=\sum_{k=0}^{\infty}(a+2k)\frac{(b)_k(c)_k(d)_k(e)_k}{(1+a-b)_{k}(1+a-c)_{k}(1+a-d)_{k}(1+a-e)_{k}},
\label{equation-wei-o}
\end{align}
where  $\mathfrak{R}(1+2a-b-c-d-e)>0$ and
\begin{align*}
E_k(a,b,c,d,e)&=\frac{(1+2a-b-c-d+2k)(a-e+k)}{1+2a-b-c-d-e+k}
\\[1mm]
&\quad+\frac{(1+a-b-c+k)(1+a-b-d+k)(e+k)}{(1+a-b+2k)(1+2a-b-c-d-e+k)}.
\end{align*}

Let $(a,b,d,e)\to(\frac{3}{2},\frac{1}{2},d-c,3-d)$ in
\eqref{equation-wei-o} to obtain
\begin{align}
&\sum_{k=0}^{\infty}\bigg(\frac{-1}{4}\bigg)^k\frac{(c)_k(d-c)_k(2+c-d)_k(2-c)_{k}(d-1)_k(3-d)_{k}}
{(\frac{3}{2})_{k}^2(1)_{k}(\frac{5}{2}-c)_{k}(\frac{5}{2}+c-d)_{k}(d-\frac{1}{2})_{k}}F_k(c,d)
\notag\\[1mm]
&\:\:=\sum_{k=0}^{\infty}\frac{4k+3}{2}\frac{(\frac{1}{2})_k(c)_k(d-c)_k(3-d)_k}{(2)_{k}(\frac{5}{2}-c)_{k}(\frac{5}{2}+c-d)_{k}(d-\frac{1}{2})_{k}},
\label{equation-wei-p}
\end{align}
where
\begin{align*}
F_k(c,d)=\frac{(2d-3+2k)(7-2d+4k)}{2}+\frac{(2-c+k)(3-d+k)(2+c-d+k)}{k+1}.
\end{align*}

 The
$(a,b,d,e,f,n)\to(\frac{3}{2},d-c,3-d,\frac{1}{2},1,\infty)$ case of
\eqref{9F8} is
\begin{align}
&\sum_{k=0}^{\infty}\frac{4k+3}{2}\frac{(\frac{1}{2})_k(c)_k(d-c)_k(3-d)_k}{(2)_{k}(\frac{5}{2}-c)_{k}(\frac{5}{2}+c-d)_{k}(d-\frac{1}{2})_{k}}
\notag\\[1mm]
&\:\:
=2\sum_{k=0}^{\infty}\frac{(c-\frac{1}{2})_k(d-c-\frac{1}{2})_k(\frac{5}{2}-d)_k}{(\frac{5}{2}-c)_{k}(\frac{5}{2}+c-d)_{k}(d-\frac{1}{2})_{k}}.
\label{equation-wei-q}
\end{align}
After substituting \eqref{equation-wei-q} into
\eqref{equation-wei-p}, it is routine to see that
\begin{align}
&\sum_{k=0}^{\infty}\bigg(\frac{-1}{4}\bigg)^k\frac{(c)_k(d-c)_k(2+c-d)_k(2-c)_{k}(d-1)_k(3-d)_{k}}
{(\frac{3}{2})_{k}^2(1)_{k}(\frac{5}{2}-c)_{k}(\frac{5}{2}+c-d)_{k}(d-\frac{1}{2})_{k}}F_k(c,d)
\notag\\[1mm]
&\:\:=2\sum_{k=0}^{\infty}\frac{(c-\frac{1}{2})_k(d-c-\frac{1}{2})_k(\frac{5}{2}-d)_k}{(\frac{5}{2}-c)_{k}(\frac{5}{2}+c-d)_{k}(d-\frac{1}{2})_{k}}.
\label{equation-wei-r}
\end{align}
Employ the derivative operator $\mathcal{D}_{c}$ on both sides of
\eqref{equation-wei-r} to get
\begin{align*}
&\sum_{k=0}^{\infty}\bigg(\frac{-1}{4}\bigg)^k\frac{(c)_k(d-c)_k(2+c-d)_k(2-c)_{k}(d-1)_k(3-d)_{k}}
{(\frac{3}{2})_{k}^2(1)_{k}(\frac{5}{2}-c)_{k}(\frac{5}{2}+c-d)_{k}(d-\frac{1}{2})_{k}}
\notag\\[1mm]
&\quad\times\Big\{
\big[H_{k}(c-1)-H_{k}(d-c-1)+H_{k}(1+c-d)-H_{k}(1-c)
\notag\\[1mm]
&\qquad\quad
 +H_{k}(\tfrac{3}{2}-c)-H_{k}(\tfrac{3}{2}+c-d)\big]F_k(c,d)+\mathcal{D}_{c}F_k(c,d)\Big\}
\notag\\[1mm]
&\:\:=2\sum_{k=0}^{\infty}\frac{(c-\frac{1}{2})_k(d-c-\frac{1}{2})_k(\frac{5}{2}-d)_k}{(\frac{5}{2}-c)_{k}(\frac{5}{2}-c+d)_{k}(d-\frac{1}{2})_{k}}
\notag\\[1mm]
&\quad\times
\Big\{H_{k}(c-\tfrac{3}{2})-H_{k}(d-c-\tfrac{3}{2})+H_{k}(\tfrac{3}{2}-c)-H_{k}(\tfrac{3}{2}+c-d)\Big\}.
\end{align*}
Dividing both sides by $(d-2c)$ and  then fixing $c=1$, we are led
to
\begin{align}
&\sum_{k=0}^{\infty}\bigg(\frac{-1}{4}\bigg)^k\frac{(1)_k(d-1)_k^2(3-d)_k^2}
{(\frac{3}{2})_{k}^3(d-\frac{1}{2})_{k}(\frac{7}{2}-d)_{k}}\bigg\{\frac{3-d+k}{k+1}+\frac{10k^2+14k-2d^2+8d-3}{2}
\notag\\[1mm]
&\quad\times
\bigg[\sum_{i=1}^k\frac{1}{i(d-2+i)}+\sum_{i=1}^k\frac{1}{i(2-d+i)}
-\sum_{i=1}^k\frac{1}{(\frac{1}{2}+i)(\frac{5}{2}-d+i)}\bigg]\bigg\}
\notag\\[1mm]
&\:\:=2\sum_{k=0}^{\infty}\frac{(\frac{1}{2})_k(d-\frac{3}{2})_k(\frac{5}{2}-d)_k}{(\frac{3}{2})_{k}(\frac{7}{2}-d)_{k}(d-\frac{1}{2})_{k}}
\notag\\[1mm]
&\quad\times
\bigg\{\sum_{i=1}^k\frac{1}{(-\frac{1}{2}+i)(d-\frac{5}{2}+i)}
-\sum_{i=1}^k\frac{1}{(\frac{1}{2}+i)(\frac{5}{2}-d+i)}\bigg\}.
\label{equation-wei-s}
\end{align}
Applying the derivative operator  $\mathcal{D}_{d}|_{d=2}$ on both
sides of \eqref{equation-wei-s} and utilizing Lemma \ref{lemm-a}, we
catch hold of
\begin{align}
&\sum_{k=0}^{\infty}\bigg(\frac{-1}{4}\bigg)^k\frac{(1)_k^5}{(\frac{3}{2})_k^5}\bigg\{(10k^2+14k+5)\Big[8H_{2k+1}^{(3)}-H_{k}^{(3)}-8\Big]+\frac{2}{k+1}\bigg\}
\notag\\[1mm]
&\:\: =\sum_{k=0}^\infty
\frac{8}{(2k+1)^3}\bigg\{8H_{2k}^{(3)}-H_k^{(3)}-4+\frac{4}{(2k+1)^3}\bigg\}
\notag\\[1mm]
&\:\: = \frac{49\zeta(3)^2}{2}-28\zeta(3).
 \label{equation-wei-t}
\end{align}
So the linear combination of \eqref{Guillera} and
 \eqref{equation-wei-t} gives
\eqref{equation-wei-b}.
\end{proof}

\section{Proof of Theorems \ref{thm-c} and \ref{thm-d}}
Firstly, we shall prove Theorem \ref{thm-c}.

\begin{proof}[{\bf{Proof of Theorem \ref{thm-c}}}]
A known transformation formula for hypergeometric series (cf.
\cite[Theorem 31]{Chu-b}) can be stated as
\begin{align}
&\sum_{k=0}^{\infty}(-1)^k\frac{(b)_k(c)_k(d)_k(e)_k(1+a-b-c)_k(1+a-b-d)_{k}(1+a-b-e)_{k}}{(1+a-b)_{2k}(1+a-c)_{2k}(1+a-d)_{2k}(1+a-e)_{2k}}
\notag\\[1mm]
&\quad\times\frac{(1+a-c-d)_k(1+a-c-e)_{k}(1+a-d-e)_{k}}{(1+2a-b-c-d-e)_{2k}}G_k(a,b,c,d,e)
\notag\\[1mm]
&\:=\sum_{k=0}^{\infty}(a+2k)\frac{(b)_k(c)_k(d)_k(e)_k}{(1+a-b)_{k}(1+a-c)_{k}(1+a-d)_{k}(1+a-e)_{k}},
\label{equation-wei-aa}
\end{align}
where $\mathfrak{R}(1+2a-b-c-d-e)>0$ and
\begin{align*}
&G_k(a,b,c,d,e)\\[1mm]
&\:=\frac{(1+2a-b-c-d+3k)(a-e+2k)}{1+2a-b-c-d-e+2k}+\frac{(e+k)(1+a-b-c+k)}{(1+a-b+2k)(1+a-d+2k)}
\\[1mm]
&\quad\times\frac{(1+a-b-d+k)(1+a-c-d+k)(2+2a-b-d-e+3k)}{(1+2a-b-c-d-e+2k)(2+2a-b-c-d-e+2k)}
\\[1mm]
&\:+\frac{(c+k)(e+k)(1+a-b-c+k)(1+a-b-d+k)}{(1+a-b+2k)(1+a-c+2k)(1+a-d+2k)(1+a-e+2k)}
\\[1mm]
&\quad\times\frac{(1+a-b-e+k)(1+a-c-d+k)(1+a-d-e+k)}{(1+2a-b-c-d-e+2k)(2+2a-b-c-d-e+2k)}.
\end{align*}

Take $(c,d,e)\to(c-b,d-c,-n)$ in \eqref{equation-wei-aa} to
establish
\begin{align*}
&\sum_{k=0}^{n}(-1)^k\frac{(b)_k(c-b)_k(d-c)_k(1+a-c)_k(1+a+b-d)_k(1+a-b+c-d)_{k}}{(1+a-b)_{2k}(1+a+b-c)_{2k}(1+a+c-d)_{2k}}
\notag\\[1mm]
&\quad\times\frac{(1+a-b+n)_k(1+a+b-c+n)_{k}(1+a+c-d+n)_{k}(-n)_{k}}{(1+a+n)_{2k}(1+2a-d+n)_{2k}}H_k(a,b,c,d,n)
\notag\\[1mm]
&\:=\sum_{k=0}^{n}(a+2k)\frac{(b)_k(c-b)_k(d-c)_k(-n)_k}{(1+a-b)_{k}(1+a+b-c)_{k}(1+a+c-d)_{k}(1+a+n)_{k}},
\end{align*}
where
\begin{align*}
&H_k(a,b,c,d,n)\\[1mm]
&\:=\frac{(1+2a-d+3k)(a+n+2k)}{1+2a-d+n+2k}+\frac{(1+a-c+k)(1+a+b-d+k)}{(1+a-b+2k)(1+a+c-d+2k)}
\\[1mm]
&\quad\times\frac{(1+a-b+c-d+k)(2+2a-b+c-d+n+3k)(k-n)}{(1+2a-d+n+2k)(2+2a-d+n+2k)}
\\[1mm]
&\:+\frac{(c-b+k)(1+a-c+k)(1+a+b-d+k)(1+a-b+c-d+k)}{(1+a-b+2k)(1+a+b-c+2k)(1+a+c-d+2k)}
\\[1mm]
&\quad\times\frac{(1+a-b+n+k)(1+a+c-d+n+k)(k-n)}{(1+a+n+2k)(1+2a-d+n+2k)(2+2a-d+n+2k)}.
\end{align*}
Employ the derivative operator $\mathcal{D}_{b}$ on both sides of it
to find
\begin{align}
&\sum_{k=0}^{n}(-1)^k\frac{(b)_k(c-b)_k(d-c)_k(1+a-c)_k(1+a+b-d)_k(1+a-b+c-d)_{k}}{(1+a-b)_{2k}(1+a+b-c)_{2k}(1+a+c-d)_{2k}}
\notag\\[1mm]
&\quad\times\frac{(1+a-b+n)_k(1+a+b-c+n)_{k}(1+a+c-d+n)_{k}(-n)_{k}}{(1+a+n)_{2k}(1+2a-d+n)_{2k}}
\notag
\end{align}
\begin{align}
 &\quad\times
\Big\{\Big[H_{k}(b-1)-H_{k}(c-b-1)+H_{k}(a+b-d)-H_{k}(a-b+c-d)
\notag\\[1mm]
&\qquad+H_{2k}(a-b)-H_{2k}(a+b-c)+H_{k}(a+b-c+n)-H_{k}(a-b+n)\Big]
\notag\\[1mm]
&\qquad\times H_k(a,b,c,d,n)+\mathcal{D}_{b}H_k(a,b,c,d,n)
 \Big\}
\notag\\[1mm]
&\:=\sum_{k=0}^{n}(a+2k)\frac{(b)_k(c-b)_k(d-c)_k(-n)_k}{(1+a-b)_{k}(1+a+b-c)_{k}(1+a+c-d)_{k}(1+a+n)_{k}}
\notag\\[1mm]
&\quad\times
\Big\{H_{k}(b-1)-H_{k}(c-b-1)+H_{k}(a-b)-H_{k}(a+b-c)\Big\}.
\label{equation-wei-bb}
\end{align}
Dividing both sides by $(c-2b)$ and  then setting $(a,b,d)=(2,1,3)$,
there is
\begin{align*}
&\sum_{k=0}^{n}(-1)^k\frac{(1)_k^2(c-1)_k^2(3-c)_k^2(c+n)_k(4-c+n)_k(2+n)_k(-n)_k}
{(2)_{2k}(c)_{2k}(4-c)_{2k}(2+n)_{2k}(3+n)_{2k}}
\notag\\[1mm]
&\quad\times
\bigg\{\bigg[\sum_{i=1}^k\frac{2}{i(c-2+i)}-\sum_{i=1}^{2k}\frac{1}{(1+i)(3-c+i)}
+\sum_{i=1}^k\frac{1}{(1+n+i)(3-c+n+i)}\bigg]
\notag\\[1mm]
&\qquad\times P_k(c,n)+Q_k(c,n)
 \bigg\}
\notag\\[1mm]
&\:\:=2\sum_{k=0}^{n}\frac{(c-1)_k(3-c)_k(-n)_k}{(4-c)_{k}(c)_{k}(3+n)_{k}}
\bigg\{\sum_{i=1}^k\frac{1}{i(c-2+i)}
-\sum_{i=1}^k\frac{1}{(1+i)(3-c+i)}\bigg\},
\end{align*}
where
\begin{align*}
P_k(c,n)
&\:=(3k+2)+\frac{(c-1+k)(3-c+k)(2+c+3k+n)(k-n)}{2(c+2k)(2+2k+n)(3+2k+n)}
\\[1mm]
&\quad+\frac{(c-1+k)^2(3-c+k)(c+k+n)(2+k+n)(k-n)}{2(c+2k)(4-c+2k)(2+2k+n)(3+2k+n)^2},
\\[1mm]
Q_k(c,n) &\:=\frac{(3-c+k)^2(k-n)}{4(c+2k)(4-c+2k)^2}
\\[1mm]
&\quad\times\frac{(64+22c-5c^2+118k+2ck+43k^2+54n+36kn+9n^2)}{(2+2k+n)(3+2k+n)^2}.
\end{align*}
Applying the derivative operator $\mathcal{D}_{c}|_{c=2}$ on both
sides of the last equation and then letting $n\to\infty$, we have
\begin{align}
&\sum_{k=0}^{\infty}\bigg(\frac{1}{64}\bigg)^k\frac{(1)_k^3}{(\frac{3}{2})_k^3}\bigg\{(21k+13)\Big[H_{2k+1}^{(3)}+2H_{k}^{(3)}-1\Big]-\frac{27}{8(k+1)^2}\bigg\}
\notag\\[1mm]
 &\:\:=16\sum_{k=0}^{\infty}\frac{(-1)^k}{(k+1)^2}\bigg\{2H_{k+1}^{(3)}-\frac{1}{(k+1)^3}-1\bigg\}
\notag\
\end{align}
\begin{align}
 &\:\:=4\pi^2\zeta(3)-36\zeta(5)-\frac{4\pi^2}{3}.
 \label{equation-wei-cc}
\end{align}
At the last step, we have used the simple identity (cf.
\cite[Equation (4.2d)]{Zheng}):
\begin{align*}
\sum_{k=0}^{\infty}\frac{(-1)^k}{(k+1)^2}H_{k+1}^{(3)}=\frac{\pi^2}{8}\zeta(3)-\frac{21}{32}\zeta(5).
\end{align*}
Then the sum of \eqref{Zeilberger} and \eqref{equation-wei-cc}
engenders \eqref{equation-wei-c}.
\end{proof}

Secondly, we start to prove Theorem \ref{thm-d}.

\begin{proof}[{\bf{Proof of Theorem \ref{thm-d}}}]
Dividing both sides of \eqref{equation-wei-bb} by $(c-2b)$ and  then
choosing $(a,b,d)=(\frac{1}{2},\frac{1}{2},\frac{3}{2})$, there
holds
\begin{align*}
&\sum_{k=0}^{n}(-1)^k\frac{(\frac{1}{2})_k^2(c-\frac{1}{2})_k^2(\frac{3}{2}-c)_k^2(c+n)_k(2-c+n)_k(1+n)_k(-n)_k}
{(1)_{2k}(c)_{2k}(2-c)_{2k}(\frac{1}{2}+n)_{2k}(\frac{3}{2}+n)_{2k}}
\notag\\[1mm]
&\quad\times
\bigg\{\bigg[\sum_{i=1}^k\frac{2}{(-\frac{1}{2}+i)(c-\frac{3}{2}+i)}-\sum_{i=1}^{2k}\frac{1}{i(1-c+i)}
+\sum_{i=1}^k\frac{1}{(n+i)(1-c+n+i)}\bigg]
\notag\\[1mm]
&\qquad\times R_k(c,n)+S_k(c,n)
 \bigg\}
\notag\\[1mm]
&\:\:=\sum_{k=0}^{n}\frac{4k+1}{2}\frac{(\frac{1}{2})_k(c-\frac{1}{2})_k(\frac{3}{2}-c)_k(-n)_k}
{(1)_k(2-c)_{k}(c)_{k}(\frac{3}{2}+n)_{k}}
\bigg\{\sum_{i=1}^k\frac{1}{(-\frac{1}{2}+i)(c-\frac{3}{2}+i)}-\sum_{i=1}^{k}\frac{1}{i(1-c+i)}\bigg\},
\end{align*}
where
\begin{align*}
R_k(c,n)
&\:=\frac{6k+1}{2}+\frac{(2c-1+2k)(3-2c+2k)(1+c+3k+n)(k-n)}{2(c+2k)(1+4k+2n)(3+4k+2n)}
\\[1mm]
&\quad+\frac{(2c-1+2k)^2(3-2c+2k)(c+k+n)(1+k+n)(k-n)}{2(c+2k)(2-c+2k)(1+4k+2n)(3+4k+2n)^2},
\\[1mm]
S_k(c,n) &\:=\frac{(3-2c+2k)^2(k-n)}{2(c+2k)(2-c+2k)^2}
\\[1mm]
&\quad\times\frac{(16+11c-5c^2+59k+2ck+43k^2+27n+36kn+9n^2)}{(1+4k+2n)(3+4k+2n)^2}.
\end{align*}
Employing the derivative operator $\mathcal{D}_{c}|_{c=1}$ on both
sides of the above formula and then letting $n\to\infty$, we are led
to
\begin{align}
&\sum_{k=0}^{\infty}\bigg(\frac{1}{64}\bigg)^k\frac{(\frac{1}{2})_k^3}{(1)_k^3}\bigg\{(42k+5)\Big[17H_{2k}^{(3)}-2H_{k}^{(3)}\Big]-\frac{27}{(2k+1)^2}\bigg\}
\notag\\[1mm]
 &\:\:=64\sum_{k=0}^{\infty}(-1)^k(4k+1)\frac{(\frac{1}{2})_k^3}{(1)_k^3}H_{2k}^{(3)}
\notag\
\end{align}
\begin{align}
\notag &\:\:=\frac{240}{\pi}\zeta(3)-128G.
\end{align}
In the final step, we have utilized the known identity:
\begin{align*}
\sum_{k=1}^{\infty}(-1)^k(4k+1)\frac{(\frac{1}{2})_k^3}{(1)_k^3}H_{2k}^{(3)}=\frac{15\zeta(3)}{4\pi}-2G,
\end{align*}
which was conjectured by  Sun \cite[Equation (3.50)]{Sun-c} and
confirmed by Wei and Xu \cite[Theorem 1.1]{Wei-Xu}.
\end{proof}

\textbf{Acknowledgments}\\

The work is supported by  the National Natural Science Foundation of
China (No. 12071103).


\end{document}